\documentclass[leqno,12pt]{amsart}
\usepackage{amsfonts}
\usepackage{amsmath,amssymb,amsthm}

\everymath{\displaystyle}
\newtheorem{thm}{Theorem}[section]
\newtheorem{cor}[thm]{Corollary}
\newtheorem{lem}[thm]{Lemma}
\newtheorem{claim}{Claim}
\newtheorem{prop}[thm]{Proposition}
\newtheorem{proposition}[thm]{Proposition}

\theoremstyle{definition}
\newtheorem{definition}[thm]{Definition}

\theoremstyle{remark}
\newtheorem{rem}[thm]{Remark}

\numberwithin{claim}{thm}
\numberwithin{equation}{section}

\newenvironment{proofofclaim}{\textit{Proof of Claim \theclaim}.}{\par Hence, the proof of the Claim \theclaim\ is completed.\hfill\ensuremath{\blacksquare}}

\begin{document}

\title[Surfaces with a canonical principal direction in $\mathbb E^3_1$]
{Complete classification of  surfaces with a canonical principal direction in the Minkowski 3-space}

\author{Alev Kelleci}
\address{F\i rat University, Faculty of Science, Department of Mathematics, 23200 Elaz\i\u g, Turkey.}
\email{alevkelleci@hotmail.com}

\author{Mahmut Erg\" ut}
\address{Nam\i k Kemal University, Faculty of Science and Letters, Department of Mathematics, 59030 Tekirda\u g, Turkey.}
\email{mergut@nku.edu.tr}

\author{Nurettin Cenk Turgay}
\address{Istanbul Technical University, Faculty of Science and Letters, Department of  Mathematics, 34469 Maslak, Istanbul, Turkey.}
\email{turgayn@itu.edu.tr}

\subjclass[2010]{Primary 53B25, Secondary 53A35, 53C50}
\keywords{Minkowski space, Lorentzian surfaces, canonical principal direction}

\begin{abstract}
In this paper, we characterize and classify all surfaces endowed with canonical principal direction relative to a space-like and light-like, constant direction in Minkowski 3-spaces.  
\end{abstract}

\maketitle

\section{Introduction}\label{sec:1}
It is well known that, a helix is a curve whose tangent lines make a constant angle with a fixed vector.   After the question `Are there  any surface making a constant angle with some fixed vector direction?' was introduced in \cite{DFJvK2007}, the concept of constant angle surfaces,  called  also as helix surfaces, have been studied geometers. Firstly, the applications of concerning surfaces in the theory of liquid crystals and of layered fluids were considered in \cite{CS2007}. They used for their study of surfaces the Hamilton-Jacobi equation, correlating the surface and  and the direction field. Further, Munteanu and Nistor gave another approach to classify all surfaces for which the unit normal makes a constant angle with a fixed direction in \cite{MunteanuNistor2009}. Moreover, the study of constant angle surfaces was extended in different ambient spaces, e.g. in $\mathbb S^{2}\times \mathbb R$ \cite{DFJvK2007} and  $\mathbb H^{2}\times \mathbb R$ \cite{DillenMunt2009}, in  $\mathbb E^{3}_1 $ \cite{LopezMunt2011,GSK2011,FuNistor2013}. In higher dimensional Euclidean space, hypersurfaces whose tangent space makes constant angle with a fixed direction are studied and a local description of how these hypersurfaces are constructed is given. They are called helix hypersurfaces, \cite{SH2009}. 

One of common geometrical properties of this type of surfaces is the following. If we denote by $U^T$ the projection of the fixed direction $k$ on the tangent plane of the surface, then $U^T$ is a principal direction of the surface with the corresponding principal curvature 0. Because of this reason, a recent natural problem that appears in the context of constant angle surfaces is to study those surfaces for which $U^T$ remains a principal direction but the corresponding principal curvature is different from zero. 

Let $N$ be a  (semi-)Riemannian manifold,  $M$ a hypersurface of $N$ and $X$ a vector field tangent to $N$.  $M$ is said to have a canonical principal direction (CPD) relative to $X$  if the tangential projection of $X$ to  $M$ gives a principal direction, \cite{GarnicaPalnas}.  One of the most common examples of hypersurfaces with CPD is rotational hypersurfaces in Euclidean spaces which have canonical principal direction relative to a vector field parallel to its rotation axis. We also want to note that  a hypersurface in an Euclidean space  with CPD relative to its position vector is said to be a generalized constant ratio hypersurface, \cite{EKTGCR,FM2014}.

The problems of  classifying hypersurfaces with CPD relative to a fixed direction $k$ have been studied by some authors recently. For example, in \cite{DFJvK2009}, this problem was studied in $\mathbb S^{2} \times\mathbb {R}$ by Dillen et. al.  Further, surfaces with CPD in $\mathbb H^{2} \times\mathbb {R}$ was studied in  \cite{DMuntNistr2011}. On these two papers $k$ was chosen to be a unit vector tangent to the second factor. On the other hand, classification results on surfaces in semi-Euclidean spaces  with CPD to  a chosen relative direction was studied in \cite{FuNistor2013,MN2011,Nistor2013}. Before we proceed, we also would like to note that when the codimension of the submanifold is more than one, a generalization of this notion was given by Tojeiro in \cite{To} and a further study appear in \cite{MT}. 

In the present paper, we would like to move the study of CPD hypersurfaces in Euclidean spaces initiated in \cite{MN2011} into semi-Euclidean spaces by obtaining partial classification of CPD surfaces in Minkowski 3-space studied in \cite{Nistor2013,FuNistor2013}. This paper is organized as follows. In Sect. 2, we introduce the notation that we will use and give a brief summary of basic definitions in theory of submanifolds of semi-Euclidean spaces. In Sect. 3, we obtain some new characterizations and the complete classification of space-like and Lorentzian CPD surfaces relative to a space-like and light-like, constant direction in the Minkowski 3-space.

\section{CPD Hypersurfaces in Minkowski spaces}\label{SectSurvey}
In this section after we give some basic equations and facts on hypersurfaces in Minkowski spaces, we would like to consider geometrical properties of hypersurfaces in a Minkowski space $\mathbb{E}^{3}_1$ endowed with a  canonical principal direction.

\subsection{Basic facts and definitions}\label{sec:2}
First, we would like to give a brief summary of basic definitions, facts and equations in the theory of  submanifolds of pseudo-Euclidean space (see for detail, \cite{ONeillKitap,ChenPRGeom2001}).

Let $\mathbb E^m_1$ denote the Minkowski $m$-space with the canonical
Lorentzian metric tensor  given by
$$
\tilde g=\sum\limits_{i=1}^{m-1} dx_i^2- dx_m^2,
$$
where $x_1, x_2, \hdots, x_m$  are rectangular coordinates of the
points of $\mathbb E^m_1$. We denote the Levi-Civita connection of $\mathbb{E}_{n+1}^{1}$ by and $\widetilde{\nabla }$.

The causality of a vector in a Minkowski space is defined as following.
A non-zero vector $v$ in $\mathbb{E}^m_1$ is said to be space-like, time-like
and light-like (null) regarding to $\left\langle v,v\right\rangle >0$ , $%
\left\langle v,v\right\rangle <0$ and $\left\langle v,v\right\rangle =0$,
respectively. Note that $v$ is said to be causal if it is not space-like.

Let $M$ be an oriented hypersurface in $\mathbb{E}_{n+1}^{1}$, $N$ and $\nabla$  its unit
normal vector associated with its orientation and  Levi-Civita connection, respectively. Then, Gauss and Weingarten
formulas are given by 
\begin{align}\nonumber
\begin{split}
\widetilde{\nabla }_{X}Y =&\nabla _{X}Y+h\left( X,Y\right) ,\\
\widetilde{\nabla }_{X}N =&-S(X),
\end{split}
\end{align}
respectively whenever $X,Y$ are tangent to $M$, where $h$ and $S$ are the second
fundamental form and the shape operator (or Weingarten map) of $M$. Note that $M$ is said to be space-like (resp.
time-like) if the induced metric $g=\left.\widetilde g\right|_{M}$ of $M$ is
Riemannian (resp. Lorentzian). This is equivalent to being time-like (resp.
space-like) of $N$ at each point of $M$.

The Codazzi equations is given by 
\begin{equation}
\label{MlinkCodazzi}(\widetilde{\nabla }_{X}h)(Y,Z) =(\widetilde{\nabla }_{Y}h)(X,Z),
\end{equation}%
where $R$ is the curvature tensor associated with the connection $\nabla $
and $\widetilde{\nabla }h$ is defined by 
$$
(\widetilde{\nabla }_{X}h)(Y,Z)=\nabla _{X}^{\perp }h(Y,Z)-h(\nabla
_{X}Y,Z)-h(Y,\nabla _{X}Z).
$$

If $M$ is space-like, then its shape operator $S$
is diagonalizable, i.e., there exists a local orthonormal frame field $%
\{e_{1},e_{2}\}$ of the tangent bundle of $M$ such that $Se_{i}=k_{i}e_{i},\quad i=1,2,%
\hdots,n$. In this case, the vector field $e_{i}$ and smooth function $k_{i}$
are called a principal direction and a principal curvature of $M$.

On the other hand, if $M$ is time-like, then by choosing an appropriated frame field  of the tangent bundle of $M$, $S$ can be assumed to have one of the following three matrix representations  
\begin{align} \label{SOPCASES}
\begin{split}
\mbox{Case I. }S=\left( \begin{array}{ccc}
k_1 &\\
&\ddots&\\
&&k_n
\end{array}\right),&\quad 
\mbox{Case II. }S=\left( \begin{array}{ccccc}
k_1 &1\\
0&k_1&&&\\
&&k_3&&\\
&&&\ddots&\\
&&&&k_n
\end{array}\right),\\
\mbox{Case III. }S=\left( \begin{array}{ccccc}
k_1 &\nu&&&\\
-\nu&k_1&&&\\
&&k_3&&\\
&&&\ddots&\\
&&&&k_n
\end{array}\right), &
\mbox{ Case IV. }S=\left( \begin{array}{cccccc}
k_1 &1&1&&&\\
1&k_1&1&&&\\
1&1&k_1&&&\\
&&&k_4&&\\
&&&&\ddots&\\
&&&&&k_n
\end{array}\right)
\end{split}
\end{align}
for some {smooth} functions $k_1,k_2,\hdots,k_n, \nu$ (see for example \cite{Magid1985}). We would like to note that in Case I and Case III of \eqref{SOPCASES}, the frame field $\{e_1,e_2\}$ is orthonormal, i.e. 
$$\langle e_1,e_1 \rangle=-1, \langle e_2,e_2 \rangle=1,\quad \langle e_i,e_j \rangle=0\mbox{ whenever $i\neq j$}$$
and it is pseudo-orthonormal in Case II and Case IV with 
$$\langle e_A,e_B \rangle=\delta_{AB}-1,\quad \langle e_1,e_A \rangle=\langle e_2,e_A \rangle=0,\quad .\langle e_i,e_j \rangle=\delta_{ij}\mbox{whenenver $A,B=1,2$, $i,j>2$}.$$

Now, let $M$ be a surface in the Minkowski 3-space. Then, its mean curvature and Gaussian curvature are defined by $H=\mathrm{trace}\, S$ and $H=\mathrm{det}\, S$, respectively. $M$ is said to be flat if $K$ vanishes identically. On the other hand, if $H=0$ and $M$ is space-like, then it is called maximal  while a time-like surface with identically vanishing mean curvature is said to be a minimal surface.

Before we proceed to next subsection, we would like to notice the notion of 
angle  in the Minkowski 3-space (see for example \cite{EKTGCR}):

\begin{definition}\label{DefAngles1}
\label{LBLDEF23} Let $v$ and $w$ be future pointing (past pointing) time-like vectors in 
$\mathbb{E}_{1}^{3}$. Then, there is a unique non-negative real number $%
\theta $ such that%
\begin{equation*}
\left\vert \left\langle v,w\right\rangle \right\vert =\left\Vert
v\right\Vert \left\Vert w\right\Vert \cosh \theta .
\end{equation*}%
The real number $\theta$ is called the Lorentzian time-like angle between $%
v $ and $w$.
\end{definition}

\begin{definition}\label{DefAngles2}
\label{LBLDEF21} Let $v$ and $w$ be a space-like vectors in $\mathbb{E}_{1}^{3}$ that
 span a space-like vector subspace. Then, we have $\left|\left\langle  v,w\right\rangle\right|%
 \leq \left\|v\right\|\left\|w\right\|$ and hence, there is a unique real number $\theta\in[0,\pi/2]$ such that%
\begin{equation*}
\left\vert \left\langle v,w\right\rangle \right\vert =\left\Vert
v\right\Vert \left\Vert w\right\Vert \cos \theta .
\end{equation*}%
The real number $\theta$ is called the Lorentzian space-like angle between $v$
and $w$.
\end{definition}

\begin{definition}\label{DefAngles3}
\label{LBLDEF22} Let $v$ and $w$ be a space-like vectors in $\mathbb{E}_{1}^{3}$ that
 span a time-like vector subspace. Then, we have $\left|\left\langle  v,w\right\rangle\right|
 > \left\|v\right\|\left\|w\right\|$ and hence, there is a unique positive real number $\theta $ such that%
\begin{equation*}
\left\vert \left\langle v,w\right\rangle \right\vert =\left\Vert
v\right\Vert \left\Vert w\right\Vert \cosh \theta .
\end{equation*}%
The real number $\theta$ is called the Lorentzian time-like angle between $v$
and $w$.
\end{definition}

\begin{definition}\label{DefAngles4}
\label{LBLDEF24} Let $v$ be a space-like vectors and $w$ a future pointing time-like
 vector in $\mathbb{E}_{1}^{3}$. Then, there is a unique non-negative real number $\theta$ such that%
\begin{equation*}
\left\vert \left\langle v,w\right\rangle \right\vert =\left\Vert
v\right\Vert \left\Vert w\right\Vert \sinh \theta .
\end{equation*}%
The real number $\theta$ is called the Lorentzian time-like angle between $v$
and $w$.
\end{definition}

\subsection{A characterization of CPD hypersurfaces}
 First, we would like to recall the following definition (See for example\cite{FuNistor2013,Nistor2013,GarnicaPalnas}).

\begin{definition}
Let $M$ be a non-degenerated hypersurface in $\mathbb{E}^{n+1}_1$ and $\zeta$ a vector field in $\mathbb{E}^{n+1}_1$. $M$ is said to be endowed with CPD relative to $\zeta$ if its tangential component is a principle direction, i.e., $S(\zeta^T)=k_1\zeta^T$ for a smooth function $k_1$, where $\zeta^T$ denotes the tangential component of $\zeta$. In particular if $X=k$ for a fixed direction $k$ in $\mathbb{E}^{n+1}_t$, we will say that $M$ is a CPD-hypersurface.
\end{definition}

On the other hand, a surface $M$ in $\mathbb E^3$ is said to be a constant angle surface (CAS) if its  unit normal makes a constant angle with a fixed vector, \cite{MunteanuNistor2009} (see also \cite{DFJvK2007,DillenMunt2009,FuNistor2013}. Later, in \cite{GSK2011,LopezMunt2011} this definition is extended to surfaces in Minkowski spaces with obvious restrictions on the causality of the fixed vector and the normal vector because of the definition of `angle' in the Minkowski space (See Definition \ref{DefAngles1}- Definition \ref{DefAngles4}).

\begin{rem} \label{RemarkForCas} 
In fact, if the ambient space is pseudo-Euclidean, then a CAS surface is a CPD surface with corresponding principle curvature $k_1=0$ (see \cite{GSK2011,LopezMunt2011,MunteanuNistor2009}). Thus, we will exclude this case. Therefore, after this point, we will locally assume that the principle curvature $k_1$ corresponding to the principle direction of tangential part of $k$ is a non-vanishing function. 
\end{rem}

Let $M$ be a hypersurface in a Minkowski space $\mathbb E^{n+1}_1$ and $k$ be a fixed direction in  $\mathbb E^{n+1}_1$ . The fixed vector $k$ can be expressed as 
\begin{equation} \label{kdecom}
k=U+\left\langle N,N\right\rangle\left\langle k,N\right\rangle N
\end{equation}
for a tangent vector $U$. We would like to give the following new characterization of CPD surfaces different from given in \cite[Theorem 2.1]{Nistor2013} and \cite[Theorem 3.7 and Theorem 4.5]{FuNistor2013}.

\begin{proposition}
\label{PROPOPP2Ext} Let $M$ be an oriented hypersurface in the Minkowski
space $\mathbb{E}^{n+1}_1$ and $k$ be a fixed vector on the tangent plane to the surface.
Consider a unit tangent vector field $e_1$ along $U$. Then, $M$ is
a CPD hypersurface if and only if a curve $\alpha$ is a geodesic of $M$
whenever it is an integral curve of $e_1$.
\end{proposition}

\begin{proof}
We will consider three cases seperately subject to causality of $U$.

\textbf{Case I.} Let $e_1$ is time-like. Thus, we have 
\begin{equation*}
k=-\langle k,e_{1}\rangle e_{1}+\langle k,N\rangle N.
\end{equation*}%
Since $\widetilde{\nabla }_{e_{1}}k=0$, this equation yields 
\begin{equation*}
0=-\langle k,\widetilde{\nabla }_{e_{1}}e_1\rangle e_1-\left\langle k,e_1\right\rangle%
 \widetilde{\nabla }_{e_{1}}e_1 -\left\langle k,Se_{1}\right\rangle N%
-\left\langle k,N\right\rangle Se_{1}.
\end{equation*}
The tangential part of this equation yields $Se_{1}=k_{1}e_{1}$ if and only
if $\nabla _{e_{1}}e_{1}=0$ which is equivalent to being geodesic of all
integral curves of $e_{1}$.

\textbf{Case II.} Let $e_1$ is space-like. Thus, we have 
\begin{equation}  \label{PROPCASEIIDECOMOP}
k=\langle k,e_1\rangle e_1+\varepsilon\langle k,N\rangle N,
\end{equation}
where $\varepsilon$ is either 1 or -1 regarding to being time-like or
space-like of $M$, respectively.

Similar to Case I, we obtain $Se_{1}=k_{1}e_{1}$ if and only if 
$\nabla_{e_{1}}e_{1}=0$.

\textbf{Case III.} Let $e_1$ is light-like. In this case, $k$ can be decompose as
\begin{equation}  \label{PROPCASEIIIDECOMOP}
k=\phi (e_1-N),
\end{equation}
for a non-constant function  $\phi$. 

Similar to the other case, we obtain $Se_{1}=k_{1}e_{1}$ if and only if 
$\nabla_{e_{1}}e_{1}=0$.
\end{proof}

%%%%%%%%%%%%%%%%%%%%%%%%%%%%%%%%
%%%%%%%%%%%%%%%%%%%%%%%%%%%%%%%%
%%%%%%%%%%%%%%%%%%%%%%%%%%%%%%%%%%%%%%%%%%%%%%%%%%%%%%%%%%%%%%%%
%%%%%%%%%%%%%%%%%%%%%%%%%%%%%%%%
%%%%%%%%%%%%%%%%%%%%%%%%%%%%%%%%%%%%%%%%%%%%%%%%%%%%%%%%%%%%%%%%
%%%%%%%%%%%%%%%%%%%%%%%%%%%%%%%%
%%%%%%%%%%%%%%%%%%%%%%%%%%%%%%%%%%%%%%%%%%%%%%%%%%%%%%%%%%%%%%%%
%%%%%%%%%%%%%%%%%%%%%%%%%%%%%%%%
%%%%%%%%%%%%%%%%%%%%%%%%%%%%%%%%%%%%%%%%%%%%%%%%%%%%%%%%%%%%%%%%
%%%%%%%%%%%%%%%%%%%%%%%%%%%%%%%%
%%%%%%%%%%%%%%%%%%%%%%%%%%%%%%%%%%%%%%%%%%%%%%%%%%%%%%%%%%%%%%%%
%%%%%%%%%%%%%%%%%%%%%%%%%%%%%%%%
%%%%%%%%%%%%%%%%%%%%%%%%%%%%%%%%%%%%%%%%%%%%%%%%%%%%%%%%%%%%%%%%
%%%%%%%%%%%%%%%%%%%%%%%%%%%%%%%%
%%%%%%%%%%%%%%%%%%%%%%%%%%%%%%%%%%%%%%%%%%%%%%%%%%%%%%%%%%%%%%%%
%%%%%%%%%%%%%%%%%%%%%%%%%%%%%%%%
%%%%%%%%%%%%%%%%%%%%%%%%%%%%%%%%%%%%%%%%%%%%%%%%%%%%%%%%%%%%%%%%
%%%%%%%%%%%%%%%%%%%%%%%%%%%%%%%%
%%%%%%%%%%%%%%%%%%%%%%%%%%%%%%%%%%%%%%%%%%%%%%%%%%%%%%%%%%%%%%%%
%%%%%%%%%%%%%%%%%%%%%%%%%%%%%%%%
%%%%%%%%%%%%%%%%%%%%%%%%%%%%%%%%%%%%%%%%%%%%%%%%%%%%%%%%%%%%%%%%
%%%%%%%%%%%%%%%%%%%%%%%%%%%%%%%%
%%%%%%%%%%%%%%%%%%%%%%%%%%%%%%%%%%%%%%%%%%%%%%%%%%%%%%%%%%%%%%%%
%%%%%%%%%%%%%%%%%%%%%%%%%%%%%%%%
%%%%%%%%%%%%%%%%%%%%%%%%%%%%%%%%%%%%%%%%%%%%%%%%%%%%%%%%%%%%%%%%
%%%%%%%%%%%%%%%%%%%%%%%%%%%%%%%%
%%%%%%%%%%%%%%%%%%%%%%%%%%%%%%%%%%%%%%%%%%%%%%%%%%%%%%%%%%%%%%%%
%%%%%%%%%%%%%%%%%%%%%%%%%%%%%%%%
%%%%%%%%%%%%%%%%%%%%%%%%%%%%%%%%

\section{Classifications of CPD Surfaces in $\mathbb E^3_1$} \label{S:Classification}
In this section,  we want to complete classification of CPD surfaces in $\mathbb E^3_1$. We would like to note that the complete classification of surfaces endowed with canonical principal direction relative to a time-like constant direction $k=(0,0,1)$ was obtained in \cite{FuNistor2013,Nistor2013}.

\subsection{CPD surfaces relative to a space-like, constant direction.}
In this subsection, we consider  surfaces endowed  with CPD relative to a space-like, constant direction $k$. In this case, up to a linear isometry 
of $\mathbb E^3_1$, we may assume  that $k=(1,0,0)$.

First, we will assume that   $M$ is a space-like surface endowed  with CPD relative to $k=(1,0,0)$. In this case, $N$ is time-like and  \eqref{kdecom} becomes
\begin{equation}\label{ExpofS100}
k=\cosh \theta e_1+\sinh \theta N
\end{equation}
for a smooth function $\theta$. Let $e_2$ be a unit tangent vector field satisfying $\langle e_1,e_2\rangle=0$. By a simple computation considering \eqref{ExpofS100} we obtain the following lemma.

\begin{lem}\label{Case1ClassThmDiagonSpacelikekClm1}
The Levi-Civita connection $\nabla$ of $M$ is given by
\begin{subequations} \label{CASEISpacelikeLeviCivitaEq1ALL}
\begin{eqnarray}
\label{CASEISpacelikeLeviCivitaEq1a}\nabla _{e_{1}}e_{1}=\nabla _{e_{1}}e_{2}=0, &&
\\\label{CASEILeviCivitaEq1b}
\nabla _{e_{2}}e_{1}=\tanh \theta k_2e_{2}, &\quad &\nabla _{e_{2}}e_{2}=-\tanh \theta k_2e_{1}, 
\end{eqnarray}
\end{subequations}
and the matrix representation shape operator $S$ of $M$ with respect to $\{e_1,e_2\}$ is
\begin{equation}\label{CASEISpacelikeLeviCivitaEq1ShpOp}
S=\left(\begin{array}{cc}
e_1(\theta)&0\\
0&k_2
\end{array}\right)
\end{equation}
for a function $k_2$ satisfying
\begin{eqnarray}\label{ClassThmDiagonSpacelikekCod1Case1}
e_1(k_2)=\tanh \theta k_2(e_1(\theta)-k_2). 
\end{eqnarray}
Furthermore, $\theta$ satisfies
\begin{equation}\label{CASEISpacelikeLeviCivitaEq1Theta}
e_2(\theta)=0.
\end{equation}
\end{lem}

\begin{proof}
By considering \eqref{ExpofS100}, one can get
\begin{equation}\label{ApplyXExpofS100}
0=X(\cosh \theta)e_1+\cosh \theta \nabla_{X}e_1+%
\cosh \theta h(e_1,X)-\sinh \theta SX+X(\sinh \theta)N
\end{equation}
whenever $X$ is tangent to $M$. \eqref{ApplyXExpofS100} for $X=e_1$ gives
\begin{eqnarray}
\nonumber\nabla_{e_1}e_1&=&0,\\
\label{ApplyXExpofS10000Eq2b} e_1(\theta)&=& k_1
\end{eqnarray}
while \eqref{ApplyXExpofS100} for $X=e_2$ is giving
\begin{eqnarray}
\nonumber\nabla_{e_2}e_1&=& \tanh \theta k_2e_2,
\end{eqnarray}
where $e_2$ is the other principle direction of $M$ with $k_2$ is the  principle curvature $k_2$  corresponding to $e_2$. Thus, we have \eqref{CASEISpacelikeLeviCivitaEq1ALL} and \eqref{ClassThmDiagonSpacelikekCod1Case1} and \eqref{CASEISpacelikeLeviCivitaEq1Theta} and the second fundamental form of $M$ becomes 
$$h(e_1,e_1)=-k_1N,\quad h(e_1,e_2)=0,\quad\quad h(e_2,e_2)=-k_2N.$$
By considering the Codazzi equation, we obtain \eqref{ClassThmDiagonSpacelikekCod1Case1}.
\end{proof}
%%%%%%%%%%%%%%%%%%%%%%%%%%%%%%%%%%%%%%%
%%%%%%%%%%%%%%%%%%%%%%%%%%%%%%%%%%%%%%%
%%%%%%%%%%%%%%%%%%%%%%%%%%%%%%%%%%%%%%%

\begin{rem}
 Because of \eqref{ApplyXExpofS10000Eq2b}, if $e_1(\theta)\equiv0$ implies $k_1=0$. We will not consider this case because of Remark \ref{RemarkForCas}). 
\end{rem}

Now, we consider a point $p\in M$ at which $e_1(\theta)$ does not vanish.   First, we would like to prove the following lemma.

\begin{lem}\label{Case1ClassThmDiagonSpacelikekClm12}
There exists a local coordinate system  $(s,t)$ defined in a neighborhood $\mathcal N_p$ of $p$ such that the induced metric of $M$ is
\begin{equation}\label{ClassThmDiagonSpacelikekDefgEqRESCase1}
g=ds^2+m^2dt^2
\end{equation}
for a function  satisfying
\begin{equation}\label{Case1ClassThmDiagonSpacelikekDefm}
e_1(m)-\tanh \theta k_2m=0.
\end{equation}
Furthermore, the vector fields $e_1,e_2$ described above become $e_1=\partial_s$,  $\displaystyle e_2=\frac 1{m}\partial_t$ in $\mathcal N_p$.
\end{lem}

\begin{proof}
Because of \eqref{CASEISpacelikeLeviCivitaEq1ALL} we have $[e_1,e_2]=-\tanh \theta k_2e_2$ because of \eqref{CASEISpacelikeLeviCivitaEq1ALL}. Thus, if $m$ is a non-vanishing smooth function on $M$ satisfying  \eqref{Case1ClassThmDiagonSpacelikekDefm}, then we have $\displaystyle \left[e_1,me_2\right]=0$. Therefore, there exists a local coordinate system $(s,t)$ such that $e_1=\partial_s$ and  $\displaystyle e_2=\frac 1m\partial_t$. Thus, the induced metric of  $M$ is as given in \eqref{ClassThmDiagonSpacelikekDefgEqRESCase1}
\end{proof}

Now, we are ready to obtain the classification theorem.
\begin{thm}\label{ClassThmDiagonSpacelikek}
Let $M$ be an oriented space-like surface in $\mathbb E^3_1$. Then, $M$ is  endowed with a canonical principal direction relative to a space-like constant direction if and only if it is congruent to the surface given by one of the followings
\begin{enumerate}
\item A surface given by
\begin{subequations}\label{ClassThmDiagonSpacelikekSurfaALL}
\begin{equation}
\label{ClassThmDiagonSpacelikekSurfaEq1}
x(s,t)=\int^s{\cosh \theta(\tau)d\tau}\Big(1,0,0\Big)+\int^s{\sinh \theta(\tau)d\tau} \Big(0,\sinh t,\cosh t\Big)+\gamma(t)
\end{equation}
where $\gamma$ is the  $\mathbb E^3_1$-valued function given by 
\begin{equation}\label{ClassThmDiagonSpacelikeSurfaEq2}
\gamma(t)=\left(0,\int^t{\Psi(\tau)\cosh \tau d\tau,\int^t{\Psi(\tau)\sinh \tau d\tau}}\right).
\end{equation}
for a function $ \Psi\in C^{\infty}(M)$;
\end{subequations}

\item A flat surface given by
\begin{align}\label{ClassThmDiagonSpacelikekSurfbEq1}
\begin{split}
x(s,t)=&\int^s{\cosh \theta(\tau)d\tau}\Big(1,0,0\Big)+\int^s{\sinh \theta(\tau)d\tau}\Big(0,\sinh t_0,\cosh t_0\Big)\\
&+\Big(0,t\cosh t_0,t\sinh t_0\Big).
\end{split}
\end{align}
for a constant $t_0$.
\end{enumerate}
\end{thm}

\begin{proof}
In order to proof the necessary condition, we assume that $M$ is endowed with a CPD relative to $k=(1,0,0)$ with the isometric immersion $x:M\rightarrow \mathbb E^3_1$. Let $\{e_1,e_2;N\}$ is the local orthonormal frame field described before Lemma \ref{Case1ClassThmDiagonSpacelikekClm1}, $k_1,k_2$  principal curvatures of $M$ and $(s,t)$ a local coordinate system given in Lemma \ref{Case1ClassThmDiagonSpacelikekClm12}.

Note that    \eqref{Case1ClassThmDiagonSpacelikekDefm}  and \eqref{ClassThmDiagonSpacelikekCod1Case1} become
\begin{eqnarray}
\label{Case1Spacelikems} m_s-m\tanh \theta k_2=0, \\
\label{Case1SpacelikeCods}(k_2)_s=(\theta' -k_2)\tanh \theta k_2,
\end{eqnarray}
respectively and  $e_2(\theta)=0$ implies $\theta=\theta(s)$. Moreover, we have 
\begin{equation}\label{AftClm1Eq1SCase1}
e_1= x_s.
\end{equation}
By combining \eqref{Case1Spacelikems} with \eqref{CASEISpacelikeLeviCivitaEq1ShpOp}, we obtain
the shape operator $S$ of $M$  as
\begin{equation}\label{CASEISpacelikeShapeOpm1}
S=\left( 
\begin{array}{cc}
\theta' & 0 \\ 
0 & \coth \theta\frac{m_s}{m}%
\end{array}%
\right)
\end{equation}
where $'$ denotes ordinary differentiation with respect to the appropriated variable.

By combining \eqref{Case1Spacelikems} and \eqref{Case1SpacelikeCods} we obtain
$$m_{ss}-\theta'\coth \theta  m_s=0
$$
whose general solution is 
$$m(s,t)=\Psi_1(t)\int^s{\sinh \theta(\tau)d\tau}+\Psi_2(t) $$
for some smooth functions $\Psi_1,\Psi_2$.
Therefore, by re-defining $t$ properly, we may assume either
\begin{subequations}\label{Case1Spacelikemall}
\begin{equation}\label{Case1Spacelikem1} 
m(s,t)=\int^s{\sinh \theta(\tau)d\tau}+\Psi(t), \Psi\in C^\infty(M), 
\end{equation}
or
\begin{equation}
\label{Case1Spacelikem2} 
m(s,t)=1.
\end{equation}
\end{subequations}

\textbf{Case 1.} $m$ satisfies \eqref{Case1Spacelikem1}. In this case, by considering the equation \eqref{CASEISpacelikeLeviCivitaEq1ALL} with $m$ given in \eqref{Case1Spacelikem1}, we get the Levi-Civita connection of $M$ satisfies 
\begin{eqnarray} \nonumber
\nabla _{\partial_s }\partial_s &=&0.
\end{eqnarray}%
 By combining this equation with \eqref{CASEISpacelikeShapeOpm1} and using Gauss formula, we obtain
\begin{eqnarray}
\label{Case1Spacelikexss} x_{ss}&=&-\theta' N.
\end{eqnarray}
On the other hand, from the decomposition \eqref{ExpofS100}, we have $\left\langle x_s,k\right\rangle=\cosh \theta$ and  $\left\langle x_t,k\right\rangle=0.$ By considering these equations, we see that  $x$ has the form of
\begin{equation}\label{Case1Spacelikex}
x(s,t)=\left(\int^s{\cosh \theta(\tau)d\tau},x_2(s,t),x_3(s,t)\right)+\gamma(t)
\end{equation}
for a $\mathbb E^3_1$-valued smooth function $\gamma=\left(0,\gamma_{2},\gamma_{3}\right)$.
 On the other hand, by combining \eqref{AftClm1Eq1SCase1} and \eqref{Case1Spacelikexss} with \eqref{ExpofS100}, we yield 
\begin{equation}\label{Case1Spacelikex2} 
(1,0,0)=\cosh \theta x_s-\frac{\sinh \theta}{\theta'}x_{ss}.
\end{equation}

By considering \eqref{Case1Spacelikex} and $\left\langle x_s,x_s\right\rangle=1$, we solve \eqref{Case1Spacelikex2} to obtain
\begin{align}\label{Case1Spacelikenewx}
\begin{split}
x(s,t)=&\int^s{\cosh \theta(\tau)d\tau}\Big(1,0,0\Big)+\int^s{\sinh \theta(\tau)d\tau} \Big(0,\sinh \varphi(t),\cosh\varphi(t)\Big)+\gamma(t)
\end{split}
\end{align}
for a smooth function $\varphi$. Note that \eqref{Case1Spacelikenewx} implies
\begin{eqnarray}
\nonumber x_s&=& \cosh\theta(s)\Big(1,0,0\Big)+\int^s \sinh \theta(\tau)d\tau\Big(0,\sinh \varphi(t),\cosh\varphi(t)\Big),\\
\label{Case1Spacelikenewxdert} x_t&=& \varphi'(t)\int^s \sinh \theta(\tau)d\tau\Big(0,\sinh \varphi(t),\cosh\varphi(t)\Big)+\Big(0,\gamma_2'(t),\gamma_3'(t)\Big)
\end{eqnarray}
and because of $\langle x_s,x_t\rangle=0$ we have $(0,\gamma_2',\gamma_3')=h(\sinh \varphi,\cosh\varphi)$ for a smooth function $h=h(t)$.
 Therefore, \eqref{Case1Spacelikenewxdert} turns into
$$ x_t= \Big(\varphi'(t)\int^s \sinh \theta(\tau)d\tau+h(t)\Big)\Big(0,\sinh \varphi(t),\cosh\varphi(t)\Big).$$
By combining this equation with $\langle x_t,x_t\rangle=m^2$ and using \eqref{Case1Spacelikem1}, we obtain $\varphi(t)=t$ and $h(t)=\Psi(t)$  which gives \eqref{ClassThmDiagonSpacelikeSurfaEq2}. In addition,    $\varphi(t)=t$ and \label{eqref} yields \eqref{ClassThmDiagonSpacelikeSurfaEq2}. Thus, we have the Case (1) of the theorem.

\textbf{Case 2.}  $m$ is given as \eqref{Case1Spacelikem2}. In this case, the induced metric of $M$ becomes $g=ds^2+dt^2$, the Levi Civita connection of $M$ satisfies 
\begin{equation}
\nabla _{\partial_s }\partial_s =0,\quad \nabla _{\partial_s }\partial_t=0, \quad%
\nabla _{\partial_t}\partial_t=0.
\end{equation}%
and \eqref{CASEISpacelikeShapeOpm1} gives
\begin{equation} \label{CASEISpacelikeShapeOpm2}
S=\left( 
\begin{array}{cc}
\theta' & 0 \\ 
0 & 0%
\end{array}%
\right).
\end{equation}
Therefore, $x$ and $N$ satisfies
\begin{eqnarray}\nonumber
\nonumber x_{ss}=-\theta' N,\quad x_{st}=0,&& \quad x_{tt}=0.\\\nonumber
N_s=-\theta'x_s,\quad N_t=0.&&
\end{eqnarray}
A straightforward computation yields that $M$ is congruent to the surface given in Case (2) of the theorem. Hence, the proof for the necessary condition is obtained.

The poof of sufficient condition follows from a direct computation.
\end{proof}

As a direct result of  Theorem \ref{ClassThmDiagonSpacelikek}, we obtain the following classification of maximal CPD surfaces.
\begin{prop}
A maximal surface in $\mathbb E^3_1$ endowed with CPD relative to a constant, space-like direction is either an open part of a plane or congruent to the surface given by
\begin{equation} \label{Case1SMaxCPDthm}
x(s,t)=\frac1{c}\left({\sin ^{-1}(c s)},{\sqrt{1-c^2 s^2} \sinh t} ,{\sqrt{1-c^2 s^2} \cosh t}\right)
\end{equation}
for a non-zero constant $c$.

In this case the angle function $\theta$ is
\begin{equation} \label{Maximalthetas}
\theta(s)=\tanh^{-1}\left(-cs\right)
\end{equation}
\end{prop}

\begin{proof}
Let $M$ be a space-like CPD surface and assume that it is not an open part of a plane. If $M$ is maximal, then Theorem \ref{ClassThmDiagonSpacelikek}  yields that $M$ is congruent to the
surface given by  \eqref{ClassThmDiagonSpacelikekSurfaALL}. Note that the shape operator $S$ of $M$ is \eqref{CASEISpacelikeShapeOpm1} for the function $m$ satisfying \eqref{Case1Spacelikem1}.   Considering the maximality condition $\mathrm{tr}\, S=0$ and  \eqref{CASEISpacelikeLeviCivitaEq1ShpOp}, we have
$$
\theta'+\coth \theta\frac{m_s}{m}=0.
$$
Solving this equation, we get
\begin{equation} \label{Case1Spacelikethetam}
\theta(s)=\cosh^{-1}\left(\frac{1}{c m}\right)
\end{equation}
for a non-zero constant $c$. Furthermore, one can conclude from \eqref{Case1Spacelikethetam} that the function $m$ depends only on $s$. So \eqref{Case1Spacelikem1} implies $\Psi(t)=0$ which yields $m(s)=\int^s{\sinh \theta(\tau)d\tau}$ and $\gamma(t)=(0,0,0).$
Therefore, \eqref{Case1Spacelikethetam} becomes
$$\theta'=-\frac 1c \cosh^2 \theta.$$
By solving this equation, we get the expression \eqref{Maximalthetas}. By a further computation, we obtain \eqref{Case1SMaxCPDthm}. Thus, we complete the proof of theorem.
\end{proof}

In the remaining part of this section,  we will assume that $M$ is a Lorentzian surface in the Minkowski 3-space  endowed with CPD relative to $k=(1,0,0)$.

As we mentioned in the previous subsection, the shape operator $S$ of $M$ can be non-diagonalizable. In this case, we can choose a pseudo-orthonormal frame field $\{e_1,e_2\}$ of the tangent bundle such that $S$ has the matrix representation 
\begin{align} \label{SOPCASESE31}
 S=\left(\begin{array}{cc}
k_1&\mu\\
0&k_1
\end{array}\right).
\end{align}
In this case, \eqref{kdecom} becomes
\begin{equation}\label{ExpofwithnondiagonalT100}
k=e_1+N.
\end{equation}
 By a simple computation we obtain $k_1=0$.  Thus $M$ is a flat, minimal B-scroll. It is well known that it must be congruent to the surface given by 
\begin{equation}\label{ClassThmNonDiagonTimelikekSurf}
x(s,t)=\left(\frac{s^2}2+t,\frac{(2s-1)^{3/2}}3,\frac{s^2}2-s+t\right)
\end{equation}
(See for example \cite{KimTrgy2017}). Hence, we have the following result.
\begin{proposition}\label{ClassThmNonDiagonSpacelikek}
Let $M$ be an oriented Lorentzian surface in $\mathbb E^3_1$ with non-diagonalizable shape operator. If $M$ is a surface endowed with a canonical principal direction relative to a space-like constant direction, then it is congruent to the surface given by \eqref{ClassThmNonDiagonTimelikekSurf}
\end{proposition}

Now, assume that $M$ is time-like and its shape operator $S$ is diagonalizable.  Let $\{e_1, e_2\}$ be a local orthonormal frame field of the tangent bundle of $M$ and $e_1$ is proportional to $U$. Since $N$ is space-like we have two cases for subject to casuality of $e_1$.

\textbf{Case A.} $e_1$ is a space-like vector. In this case, \eqref{kdecom} implies
\begin{equation}\label{ExpofT100}
k=\sin \theta e_1+\cos \theta N.
\end{equation}

\textbf{Case B.} $e_1$ is a time-like vector. In this case, \eqref{kdecom} implies
\begin{equation}\label{ExpofTb100}
k=\sinh \theta e_1+\cosh \theta N.
\end{equation}

We have the following lemma which is the analogous of Lemma \ref{Case1ClassThmDiagonSpacelikekClm1}.

\begin{lem}\label{Case2ClassThmDiagonSpacelikekClm1}
Let $M$ be a Lorentzian surface endowed with CPD relative to $k=(1,0,0)$ and $\{e_1,e_2\}$ its principle directions such that 
$\langle k,e_2 \rangle =0$. Then we have the following statements.
\begin{enumerate}
\item If $e_1$ is space-like, then the Levi-Civita connection $\nabla$ of $M$ is given by
\begin{subequations} \label{CASEIISpacelikeLeviCivitaEq1ALL}
\begin{eqnarray}
\nabla _{e_{1}}e_{1}=\nabla _{e_{1}}e_{2}=0, &&
\label{CASEIILeviCivitaEq1a} \\\label{CASEIILeviCivitaEq1b}
\nabla _{e_{2}}e_{1}=\cot \theta k_2e_{2}, &\quad &\nabla _{e_{2}}e_{2}=\cot \theta k_2e_{1}
\end{eqnarray}
\end{subequations}
for a function $k_2$ satisfying
\begin{eqnarray}\label{ClassThmDiagonSpacelikekCod1CaseII}
e_1(k_2)=\cot \theta k_2(e_1(\theta)-k_2).
\end{eqnarray}

\item If $e_1$ is time-like, then the Levi-Civita connection $\nabla$ of $M$ is given by
\begin{subequations} \label{CASEIIbSpacelikeLeviCivitaEq1ALL}
\begin{eqnarray}
\label{CASEIIbLeviCivitaEq1a} \nabla _{e_{1}}e_{1}=\nabla _{e_{1}}e_{2}=0, &&\\
\label{CASEIIbLeviCivitaEq1b} \nabla _{e_{2}}e_{1}=\coth \theta k_2e_{2}, &\quad &\nabla _{e_{2}}e_{2}=\coth \theta k_2e_{1}, 
\end{eqnarray}
\end{subequations}
and for a function $k_2$ satisfying
\begin{eqnarray}\label{ClassThmDiagonSpacelikekCod1CaseIIb}
e_1(k_2)=\coth \theta k_2(e_1(\theta)-k_2).
\end{eqnarray}

\item In both cases $\theta$ satisfies \eqref{CASEISpacelikeLeviCivitaEq1Theta} and the matrix representation shape operator $S$  is
\begin{equation}\label{CASEIIbSpacelikeShapeOp}
S=\left( 
\begin{array}{cc}
e_1(\theta) & 0 \\ 
0 & k_2%
\end{array}%
\right).
\end{equation}
\end{enumerate}
\end{lem}
\begin{proof}
We use exactly same way with the proof of Lemma \ref{Case1ClassThmDiagonSpacelikekClm1}. By considering \eqref{ExpofT100} and \eqref{ExpofTb100}, we get the statement (1) and (2) of the lemma, respectively and obtain \eqref{CASEISpacelikeLeviCivitaEq1Theta} and \eqref{CASEIIbSpacelikeShapeOp} for both cases.
\end{proof}

The proof of the following lemma is similar to the proof of Lemma \ref{Case1ClassThmDiagonSpacelikekClm12}.
\begin{lem}\label{Case2ClassThmDiagonSpacelikekClm12}
Let $M$ be a Lorentzian surface endowed with CPD relative to $k=(1,0,0)$ and $\{e_1,e_2\}$ its principle directions such that 
$\langle k,e_2 \rangle =0$. Then there exists  a neighborhood $\mathcal N_p$ of $p$ on which $e_1=\partial_s$ and  $\displaystyle e_2=\frac 1{m}\partial_t$ for a smooth function $m$. Moreover, if $e_1$ is space-like then the induced metric of $\mathcal N_p$ becomes
\begin{equation}\label{ClassThmDiagonSpacelikekDefgEqRESCaseII}
g=ds^2-m^2dt^2
\end{equation}  
and $m$ satisfies
\begin{equation}\label{Case2ClassThmDiagonSpacelikekDefm}
e_1(m)-\cot \theta k_2m=0.
\end{equation}

On the other hand, if if $e_1$ is time-like then the induced metric of $\mathcal N_m$ becomes
\begin{equation}\label{ClassThmDiagonSpacelikekDefgEqRESCaseIIb}
g=-ds^2+m^2dt^2
\end{equation}  
and $m$ satisfies
\begin{equation}\label{Case2bClassThmDiagonSpacelikekDefm}
e_1(m)-\coth \theta k_2m=0.
\end{equation}
\end{lem}

\begin{thm}\label{ClassThmCase2DiagonSpacelikek}
Let $M$ be an oriented Lorentzian surface in $\mathbb E^3_1$ with diagonalizable shape operator. Then, $M$ is  endowed with a canonical principal direction relative to a space-like, constant direction if and only if it is congruent to the surface given by one of the followings
\begin{enumerate}
\item A surface given by
\begin{subequations}\label{ClassThmDiagonSpacelikekSurfcALL}
\begin{equation}\label{ClassThmDiagonSpacelikekSurfcEq1}
x(s,t)=\int^s{\sin \theta(\tau)d\tau}\Big(1,0,0\Big)+\int^s{\cos \theta(\tau)d\tau}\Big(0,\cosh t ,\sinh t\Big)+\gamma(t),
\end{equation}
where $\gamma$ is   
\begin{equation}\label{ClassThmDiagonSpacelikekSurfcEq2}
\gamma(t)=(0,\int^t{\Psi(\tau)\sinh(\tau)d\tau},\int^t{\Psi(\tau)\cosh(\tau)d\tau})
\end{equation}
for a function $ \Psi\in C^{\infty}(M)$;
\end{subequations}

\item A surface given by
\begin{align}\label{ClassThmDiagonSpacelikekSurfdEq1}
\begin{split}
x(s,t)=&\int^s{\sin \theta(\tau)d\tau}\Big(1,0,0\Big)+\int^s{\cos \theta(\tau)d\tau}\Big(0,\cosh t_0 ,\sinh t_0 \Big)\\
&+\Big(0,t\sinh(t_0),t\cosh(t_0)\Big)
\end{split}
\end{align}
for  a constant $t_0$;

\begin{subequations}\label{ClassThmDiagonSpacelikekSurfeALL}
\item  A surface given by
\begin{equation} \label{ClassThmDiagonSpacelikekSurfeEq1}
x(s,t)=\int^s{\sinh \theta(\tau)d\tau}\Big(-1,0,0\Big)+\int^s{\cosh \theta(\tau)d\tau}\Big(0,\sinh t,\cosh t \Big)+\gamma(t),
\end{equation}
where $\gamma$ is   
\begin{equation} \label{ClassThmDiagonSpacelikekSurfeEq2}
\gamma(t)=(0,\int^t{\Psi(\tau)\cosh(\tau)d\tau},\int^t{\Psi(\tau)\sinh(\tau)d\tau})
\end{equation}
for a function $ \Psi\in C^{\infty}(M)$;
\end{subequations}

\item A surface given by
\begin{align}\label{ClassThmDiagonSpacelikekSurffEq1}
\begin{split}
x(s,t)=&\int^s{\sinh \theta(\tau)d\tau}\Big(-1,0,0\Big)+\int^s{\cosh \theta(\tau)d\tau}\Big(0,\sinh t_0,\cosh t_0\Big)\\&
+\Big(0,t\cosh t_0,t\sinh t_0\Big),
\end{split}
\end{align}
for a constant $t_0$.
\end{enumerate}
\end{thm}

\begin{proof}
In order to prove the necessary condition, we assume that $M$ is endowed with CPD relative to $k=(1,0,0)$. Let $x:M\rightarrow \mathbb E^3_1$ be an isometric immersion, $\{e_1,e_2;N\}$  the local orthonormal frame field described before Lemma \ref{Case2ClassThmDiagonSpacelikekClm1}, $k_1,k_2$  principal curvatures of $M$ and $(s,t)$ a local coordinate system given in Lemma \ref{Case2ClassThmDiagonSpacelikekClm12}. We will consider two cases described above seperately.

\textbf{Case A.} $e_1$ is a space-like vector. In this case, we have \eqref{CASEIISpacelikeLeviCivitaEq1ALL}-\eqref{ClassThmDiagonSpacelikekCod1CaseII}, \eqref{ClassThmDiagonSpacelikekDefgEqRESCaseII} and \eqref{Case2ClassThmDiagonSpacelikekDefm}. Note that \eqref{Case2ClassThmDiagonSpacelikekDefm} and \eqref{ClassThmDiagonSpacelikekCod1CaseII} turns into 
\begin{subequations}\label{Case2ClassThmDiagonSpacelikekDefm2aa} 
\begin{eqnarray}
\label{Case2ClassThmDiagonSpacelikekDefm2} m_s-\cot \theta k_2m=0,\\
\label{ClassThmDiagonSpacelikekCod1CaseII2} (k_2)_s=\cot \theta k_2(\theta'-k_2),
\end{eqnarray}
\end{subequations}
respectively.

By considering \eqref{Case2ClassThmDiagonSpacelikekDefm2} we obtain $k_2=\tan \theta\frac{m_s}{m}.$ Thus, \eqref{CASEIIbSpacelikeShapeOp} becomes
\begin{equation}\label{CASEIISpacelikeShapeOpm1}
S=\left( 
\begin{array}{cc}
\theta' & 0 \\ 
0 &\tan \theta\frac{m_s}{m}
\end{array}%
\right).
\end{equation}
Furthermore, by differentiating \eqref{Case2ClassThmDiagonSpacelikekDefm2} with respect to $s$ and using \eqref{Case2ClassThmDiagonSpacelikekDefm2aa}, we obtain
$$m_{ss}+\theta'\tan\theta m_s=0.$$
Therefore,  $m$ satisfies either 
\begin{subequations}\label{Case2Spacelikemall}
\begin{equation}\label{Case2Spacelikem1} 
m(s,t)=\int^s{\cos \theta(\xi)d\xi}+\Psi(t) 
\end{equation}
for a smooth function $\Psi$ or
\begin{equation}
\label{Case2Spacelikem2} 
m(s,t)=1
\end{equation}
\end{subequations}

\textbf{Case A1.} $m$ satisfies \eqref{Case2Spacelikem1}. In this case, similar to the Case (1) in the proof of Theorem \ref{ClassThmDiagonSpacelikek}, we consider \eqref{CASEIISpacelikeLeviCivitaEq1ALL} and \eqref{CASEIISpacelikeShapeOpm1} to get
\begin{subequations}
\begin{eqnarray}
\label{Case2Spacelikexss} x_{ss}&=&\theta' N,\\
\label{Case2spacelikexst} x_{st}&=&\frac{m_s}{m}x_t, \\
\label{Case2Spacelikextt} x_{tt}&=& mm_s x_ s+\frac{m_t}{m}x_ t-mm_s\tan \theta N.
\end{eqnarray}
\end{subequations}
Furthermore, considering \eqref{ExpofT100} we have $\left\langle e_1,k\right\rangle=\left\langle x_s,k\right\rangle=\sin \theta$ and  $\left\langle x_t,k\right\rangle=0$. So we get
\begin{equation}\label{Case2Spacelikex}
x(s,t)=\left(\int^s{\sin \theta(\tau)d\tau},x_2(s,t),x_3(s,t)\right)+\gamma(t)
\end{equation}
for a $\mathbb E^3_1$-valued smooth function $\gamma=\left(0,\gamma_{2},\gamma_{3}\right)$.
Also \eqref{ExpofT100} and \eqref{Case2Spacelikexss} imply
\begin{equation}\label{Case2Spacelikex2} 
(1,0,0)=\sin \theta x_s+\frac{\cos \theta}{\theta'}x_{ss}.
\end{equation}

By considering \eqref{Case2Spacelikex} and $\langle x_s,x_s\rangle=1$, we solve \eqref{Case2Spacelikex2} and obtain
\begin{align}\label{Case2Spacelikenewx}
x(s,t)=&\int^s{\sin \theta(\tau)d\tau}\Big(1,0,0\Big)+\int^s{\cos \theta(\tau)d\tau}\Big(0,\cosh \varphi(t) ,\sinh \varphi(t)\Big)+\gamma(t),
\end{align}
for a  smooth function $\varphi$. By a similar way  in the Case (1) in the proof of Theorem \ref{ClassThmDiagonSpacelikek}, we could get 
$\varphi(t)=t$ and \eqref{ClassThmDiagonSpacelikekSurfcEq2} by considering \eqref{ClassThmDiagonSpacelikekDefgEqRESCaseII} and \eqref{Case2Spacelikenewx}. Furthermore, considering $\varphi(t)=t$ and \eqref{ClassThmDiagonSpacelikekSurfcEq2} in \eqref{Case2Spacelikenewx} we get 
\eqref{ClassThmDiagonSpacelikekSurfcEq1}. Hence, we get the classification of surface in the case (1) of the Theorem \ref{ClassThmCase2DiagonSpacelikek}.

\textbf{Case A2.} $m$ satisfies \eqref{Case2Spacelikem2}.  In this case,\eqref{ClassThmDiagonSpacelikekDefgEqRESCaseII} turns into $g=ds^2-dt^2$, and \eqref{CASEISpacelikeShapeOpm1} gives \eqref{CASEISpacelikeShapeOpm2}. Therefore, $x$ and $N$ satisfies
\begin{eqnarray}\nonumber
\nonumber x_{ss}=\theta' N,\quad x_{st}=0,&& \quad x_{tt}=0.\\\nonumber
N_s=-\theta'x_s,\quad N_t=0.&&
\end{eqnarray}
A straightforward computation yields that $M$ is congruent to the surface given in Case (2) of the Theorem \ref{ClassThmCase2DiagonSpacelikek}. Hence, the proof for the necessary condition is obtained.

Now, we would like to get the case (3) and the case (4) of the Theorem \ref{ClassThmCase2DiagonSpacelikek}.

\textbf{Case B.} $e_1$ is a time-like vector. In this case, we have \eqref{CASEIIbSpacelikeLeviCivitaEq1ALL}-\eqref{ClassThmDiagonSpacelikekCod1CaseIIb}, \eqref{ClassThmDiagonSpacelikekDefgEqRESCaseIIb} and \eqref{Case2bClassThmDiagonSpacelikekDefm}. By a similar way to Case A we obtain 
\begin{equation}\label{CASEIIbSpacelikeShapeOpm1}
S=\left( 
\begin{array}{cc}
\theta' & 0 \\ 
0 &\tanh \theta\frac{m_s}{m}
\end{array}%
\right).
\end{equation}
Similar to the Case A, we obtain
$$m_{ss}+\theta'\tanh\theta m_s=0.$$
which yields that $m$ satisfies either 
\begin{equation}\label{Case2bSpacelikem1} 
m(s,t)=\int^s{\cosh \theta(\xi)d\xi}+\Psi(t) 
\end{equation}
for a smooth function $\Psi$ or \eqref{Case2Spacelikem2}.

If $m$ satisfies \eqref{Case2bSpacelikem1}, we use exactly the same  way that we did in the Case A1 and obtain the Case (3) of the theorem. On the other hand, if $m(s,t)=1$, then we get the Case (4) of the theorem. Hence, the proof of the necessary condition is completed.

The proof of sufficient condition follows from a direct computation.
\end{proof}
%%%%%%%%%%%%%%%%%%%%%%%%%%%%%%%%%%%%%%%%%%%%%%%%%%%%%%%%%%%%%%%%%
%%%%%%%%%%%%%%%%%%%%%%%%%%%%%%%%%%%%%%%%%%%%%%%%%%%%%%%%%%%%%%%%%%%%%%%%%%%%%%%%%%%%%%%%%%%%%%%%%%%%%%%%%%%%%%%%%%%%%%%%%%%%%%%%%%%%%%%%%%%%%%%%%%%%%%%%%%%%%%%%%%%%%%%%%%%%%%%%%%%%%%%%%%%%%%%%%%%%%%%%%%%%%%%%%%%%%%%%%%%%%%%%%%%%%%%%%%%%%%%

\begin{prop}
A minimal surface in $\mathbb E^3_1$ endowed with CPD relative to a constant, space-like direction is either an open part of a plane or congruent to one of following two surface given below
\begin{enumerate}
\item A surface given by
\begin{equation} \label{Case1SMinCPDthm}
x(s,t)=\frac 1{c}\left(\sinh ^{-1}(c s),\sqrt{c^2 s^2+1} \cosh t, \sqrt{c^2 s^2+1} \sinh t\right)
\end{equation}
for a non-zero constant $c$. In this case, the angle function $\theta$ is
\begin{equation} \label{MinimalthetasCase1}
\theta(s)=\cot^{-1} (cs).
\end{equation}

\item A surface given by
\begin{equation} \label{Case2SMinCPDthm}
x(s,t)= \frac 1{c}\left(-\ln \left(\sqrt{c^2 s^2-1}+c s\right),\sqrt{c^2 s^2-1} \sinh t ,\sqrt{c^2 s^2-1} \cosh t \right)
\end{equation}
for a non-zero constant $c$. In this case, the angle function $\theta$ is
\begin{equation} \label{MinimalthetasCase2}
\theta(s)=\coth^{-1} (cs).
\end{equation}
\end{enumerate}

\end{prop}

\begin{proof}

Let $M$ be a Lorentzian CPD surface and assume that it is not an open part of a plane. If $M$ is minimal, then Theorem \ref{ClassThmCase2DiagonSpacelikek}  yields that $M$ is congruent to the
one of surfaces given by \eqref{ClassThmDiagonSpacelikekSurfcALL} and \eqref{ClassThmDiagonSpacelikekSurfeALL}.

\textbf{Case 1.} $M$ is congruent to the surface given by  \eqref{ClassThmDiagonSpacelikekSurfcALL}. Note that the shape operator $S$ of $M$ is \eqref{CASEIISpacelikeShapeOpm1} for the function $m$ satisfying \eqref{Case2Spacelikem1}. Then the minimality condition $\mathrm{tr} S=0$ and \eqref{CASEIISpacelikeShapeOpm1} give
$$\theta' +\tan \theta\frac{m_s}{m}=0$$
which implies
\begin{equation}\label{MinCPDEq1}
\theta(s)=\sin^{-1}\left(\frac {1}{cm}\right)
\end{equation}
for a non-zero constant $c$ and $m=m(s)$. Therefore, \eqref{Case2Spacelikem1} give $\Psi=0$. So,  
$$m(s,t)=\int^s{\cos \theta(\xi)d\xi}.$$
By combining this equation with \eqref{MinCPDEq1} we obtain \eqref{MinimalthetasCase1}. By a further computation, we obtain \eqref{Case1SMinCPDthm}.  

\textbf{Case 2.} $M$ is congruent to the surface given by  \eqref{ClassThmDiagonSpacelikekSurfeALL}. Note that the shape operator $S$ of $M$ is \eqref{CASEIIbSpacelikeShapeOpm1} for the function $m$ satisfying \eqref{Case2bSpacelikem1}. In this case, the minimality condition $\mathrm{tr} S=0$ and \eqref{CASEIISpacelikeShapeOpm1} give
$$\theta' +\tanh \theta\frac{m_s}{m}=0.$$
By a similar way to Case 1, we obtain \eqref{Case2SMinCPDthm} and \eqref{MinimalthetasCase2}.
\end{proof}

\subsection{CPD surfaces relative to a light-like, constant direction.}

In this subsection we will consider surfaces endowed with CPD relative to the fixed vector $k=(1,0,1)$ which is light-like.
\begin{thm}\label{ClassThmDiagonLightlikek}
Let $M$ be an oriented surface in $\mathbb E^3_1$ with diagonalizable shape operator. Then, $M$ is endowed with a canonical principal direction relative to a light-like, constant direction if and only if it is congruent to the surface given by 
\begin{align}\label{ClassThmDiagonLightlikekSurf}
\begin{split}
x(s,t)=& \left(\int_{s_0}^s\frac{1}{2\phi(\xi)^2}d\xi\right)\Big(1,0,1\Big)+s\left(\gamma_0(t),\sqrt{-2\varepsilon \gamma_0(t)+1},\gamma_0(t)-\varepsilon\right)\\
&+\int_{t_0}^tb(\xi)\left(\sqrt{-2\varepsilon \gamma_0(\xi)+1},-\varepsilon,\sqrt{-2\varepsilon \gamma_0(\xi)+1}\right)d\xi
\end{split}
\end{align}
for some smooth functions $b,\gamma_0$, some constants $s_0,t_0$ and $\varepsilon\in\{ -1,1\}$ and a non-vanishing function $\phi$ whose derivative does not vanish. In this case, the tangential vector field $(1,0,1)^T=\phi(s) e_1$ is a principle direction of $M$  for a vector field $\langle e_1,e_1\rangle=\varepsilon.$  
\end{thm}

\begin{proof}
Let $N$ be the unit normal vector field of $M$ associated with its orientation and $x:M\rightarrow \mathbb E^3_1$ an isometric immersion. We put $\varepsilon=-\langle N,N\rangle$. Assume that $e_1$ is the unit tangent normal vector field proportional to tangential part of $k=(1,0,1)$ and $e_2$ is a unit space-like tangent vector field with $\langle e_1,e_2\rangle=0$. Then, we have
\begin{equation}\label{Expof101}
(1,0,1)=\phi(e_1-N)
\end{equation}
for a smooth function $\phi.$ Note that we have $\langle e_1,e_1\rangle=\varepsilon$.

Now, in order to proof the necessary condition, we assume that $e_1$ is a principle direction of $M$ with corresponding principle curvature $k_1$. By a simple computation considering \eqref{Expof101} we obtain
\begin{equation}\label{ApplyX2Expof101}
0=X(\phi)(e_1-N)+\phi \nabla_{X}e_1+\phi h(e_1,X)+\phi SX
\end{equation}
whenever $X$ is tangent to $M$. Note that \eqref{ApplyX2Expof101} for $X=e_1$ gives
\begin{subequations}\label{ApplyX2Expof101Eq2ALL}
\begin{eqnarray}
\label{ApplyX2Expof101Eq2a}\nabla_{e_1}e_1&=&0,\\
\label{ApplyX2Expof101Eq2b}e_1(\phi)&=&-\phi k_1
\end{eqnarray}
while \eqref{ApplyX2Expof101} for $X=e_2$ is giving
\begin{eqnarray}
\label{ApplyX2Expof101Eq2c}\nabla_{e_2}e_1&=&-k_2e_2,\\
\label{ApplyX2Expof101Eq2d}e_2(\phi)&=&0,
\end{eqnarray}
\end{subequations}
where $e_2$ is the other principle direction of $M$ with corresponding principle curvature $k_2$ and $\langle e_2,e_2\rangle=1$. In addition, the second fundamental form of $M$ becomes
\begin{eqnarray}
\label{ClassThmDiagonLightlikeSecFundForm} h(e_1,e_1)=-k_1N,\quad h(e_1,e_2)=0,\quad\quad h(e_2,e_2)=-\varepsilon k_2N.
\end{eqnarray}
Therefore, the Codazzi equation gives
\begin{equation}\label{ClassThmDiagonLightlikekCod1}
e_1(k_2)=k_2^2-k_1k_2\quad\mbox{ and }\quad e_2(k_1)=0.
\end{equation}
Note that, because of Remark \ref{RemarkForCas}, \eqref{ApplyX2Expof101Eq2b} implies that $e_1(\phi)$ does not vanish on $M$. 

Let $p\in M$. First, we would like to prove the following claim.
%%%%%%%%%%%%%%%%%%%%%%%%%%%%%%%%%%%%%%%%%%%%%%%%%%%%%%%%%%%%%%%%%%%%%%%%%%%%%%%%%%%%%%%%%%%%%%%%%%%%%%%%%%%%%%%%%%%%%%%%%%%%%%%%%%%%%%%%%%%%%%%%%%%%%%%%%%%%%%%%%%%%%%%%%%%%%%%%%%%%%%%%%%%%%%%%%%%%%%%%%%%%%%%%%%%%%%%%%%%%%%%%%%%%%%%%%%%%%%%%%%%%%%%%%%%%%%%%%%%%%%%%%%%%%%%%%%%%%%%%%%%%%%%%%%%%%%%%%%%%%%%%%%%%%%%%%%%%%%%%%%%%%%%%%%%%%%%%%%%%%%%%%%%%%%%%%%%%%%%%%%%%%%%%%%%%%%%%%%%%%%%%%%%%%%%%%%%%%%%%%%%%%%%%%%%%%%%%%%%%%%%%%%%%%%%%%%%%%%%%%%%%%%%%%%%%%%%%%%%%%%%%%%%%%%%%%%%%%%%%

\begin{claim}\label{ClassThmDiagonLightlikekClm1}
There exists a neighborhood $\mathcal N_p$ of $p$ on which the induced metric of $M$ becomes
\begin{equation}\label{ClassThmDiagonLightlikekDefgEqRES}
g=\frac\varepsilon{\phi(s)^2}ds^2+(a(t)s+b(t))^2dt^2
\end{equation}  
for some smooth functions $a,b$ such that $e_1=\phi\partial_s$,  $\displaystyle e_2=\frac 1{a(t)s+b(t)}\partial_t$ and
\begin{equation}\label{ClassThmDiagonLightlikekDefphi}
k_1(s)=-\phi'(s)
\end{equation}
\end{claim}

\begin{proofofclaim}
Note that we have $[e_1,e_2]=k_2e_2$ because of \eqref{ApplyX2Expof101Eq2a} and \eqref{ApplyX2Expof101Eq2c}. Therefore, \eqref{ApplyX2Expof101Eq2d} implies $\displaystyle \left[\frac 1\phi e_1,Ge_2\right]=0$ for any function $G$ satisfying
\begin{equation}\label{ClassThmDiagonLightlikekDefG}
e_1(G)=-k_2G.
\end{equation}
Therefore, there exists a local coordinate system $(s,t)$ such that $e_1=\phi\partial_s$ and  $\displaystyle e_2=\frac 1G\partial_t$. Thus, the induced metric of  $M$ is
\begin{equation}\label{ClassThmDiagonLightlikekDefgEq1}
g=\frac\varepsilon{\phi^2}ds^2+G^2dt^2.
\end{equation}
Note that we have $k_1=k_1(s)$ and \eqref{ClassThmDiagonLightlikekDefphi}
because of  \eqref{ApplyX2Expof101Eq2b}, \eqref{ApplyX2Expof101Eq2d} and \eqref{ClassThmDiagonLightlikekCod1}. In addition, the first equation in \eqref{ClassThmDiagonLightlikekCod1} and \eqref{ClassThmDiagonLightlikekDefG} give
\begin{equation}\label{ClassThmDiagonLightlikekCod1Eq1a}
\phi(k_2)_s=k_2(k_2-k_1)
\end{equation}
and
\begin{equation}\label{ClassThmDiagonLightlikekDefGv2}
\phi(s) G_s=-k_2G
\end{equation}
respectively.
Now, getting derivative of \eqref{ClassThmDiagonLightlikekDefG} implies
\begin{equation}\label{ClassThmDiagonLightlikekDefGss}
\phi'G_s+\phi G_{ss}=-(k_2)_sG-k_2G_s.
\end{equation}
By combining \eqref{ClassThmDiagonLightlikekDefGv2},  \eqref{ClassThmDiagonLightlikekDefphi} and \eqref{ClassThmDiagonLightlikekCod1Eq1a} with \eqref{ClassThmDiagonLightlikekDefGss}, we obtain $\phi G_{ss}=0$ which yields $G=a(t)s+b(t)$ for some smooth functions $a,b$. Therefore, \eqref{ClassThmDiagonLightlikekDefgEq1} becomes \eqref{ClassThmDiagonLightlikekDefgEqRES}.
\end{proofofclaim}
%%%%%%%%%%%%%%%%%%%%%%%%%%%%%%%%%%%%%%%%%%%%%%%%%%%%%%%%%%%%%%%%%%%%%%%%%%%%%%%%%%%%%%%%%%%%%%%%%%%%%%%%%%%%%%%%%%%%%%%%%%%%%%%%%%%%%%%%%%%%%%%%%%%%%%%%%%%%%%%%%%%%%%%%%%%%%%%%%%%%%%%%%%%%%%%%%%%%%%%%%%%%%%%%%%%%%%%%%%%%%%%%%%%%%%%%%%%%%%%%%%%%%%%%%%%%%%%%%%%%%%%%%%%%%%%%%%%%%%%%%%%%%%%%%%%%%%%%%%%%%%%%%%%%%%%%%%%%%%%%%%%%%%%%%%%%%%%%%%%%%%%%%%%%%%%%%%%%%%%%%%%%%%%%%%%%%%%%%%%%%%%%%%%%%%%%%%%%%%%%%%%%%%%%%%%%%%%%%%%%%%%%%%%%%%%%%%%%%%%%%%%%%%%%%%%%%%%%%%%%%%%%%%%%%%%%%%%%%%%%

Now, let $s,t$ be local coordinates described in the Claim \ref{ClassThmDiagonLightlikekClm1}. Note that we have
\begin{equation}\label{AftClm1Eq1}
e_1=\phi x_s.
\end{equation}
Moreover, \eqref{ApplyX2Expof101Eq2a} and \eqref{ClassThmDiagonLightlikeSecFundForm} imply
$$\widetilde\nabla_{\phi\partial_s}(\phi\partial_s)=-k_1N$$
from which we get
\begin{equation}\label{AftClm1Eq2}
N=\frac{1}{\phi'}\left(\phi\phi'x_s+\phi^2x_{ss}\right).
\end{equation}
By combining \eqref{AftClm1Eq1} and \eqref{AftClm1Eq2} with \eqref{Expof101} we get 
\begin{equation}\nonumber
(1,0,1)=\phi\left(\phi x_s-\frac{1}{\phi'}\left(\phi\phi'x_s+\phi^2x_{ss}\right)\right)
\end{equation}
which yields
$$x_{ss}=-\frac{\phi'}{\phi^3}(1,0,1).$$
By integrating this equation and considering \eqref{AftClm1Eq1}, we get
\begin{equation}\label{AftClm1Eq3}
x_s=\frac{1}{2\phi^2}(1,0,1)+\gamma(t)
\end{equation}
for an $\mathbb E^3_1$-valued smooth function $\gamma$. As $\langle x_s,x_s\rangle=\frac{\varepsilon}{\phi^2}$, we have
$$\frac{\langle(1,0,1),\gamma(t)\rangle-\varepsilon}{\phi(s)^2}+\langle\gamma(t),\gamma(t)\rangle=0.$$
Since $\phi$ is not constant, the above equation implies $\langle(1,0,1),\gamma(t)\rangle=\varepsilon$ and $\langle\gamma(t),\gamma(t)\rangle=0.$ By considering these equations, we obtain 
$$\gamma(t)=\left(\gamma_0(t),\sqrt{-2\varepsilon \gamma_0(t)+1},\gamma_0(t)-\varepsilon\right)$$ 
for a smooth function $\gamma_0$. Therefore \eqref{AftClm1Eq3} becomes
\begin{equation}\label{AftClm1Eq4}
x_s=\frac{1}{2\phi^2}(1,0,1)+\left(\gamma_0(t),\sqrt{-2\varepsilon \gamma_0(t)+1},\gamma_0(t)-\varepsilon\right).
\end{equation}
 \eqref{AftClm1Eq1} and \eqref{AftClm1Eq2} imply
\begin{subequations}\label{AftClm1Eq5ALL}
\begin{eqnarray}
\label{AftClm1Eq5a} e_1=\phi x_s&=&\frac{1}{2\phi(s)}(1,0,1)+\phi(s)\left(\gamma_0(t),\sqrt{-2\varepsilon \gamma_0(t)+1},\gamma_0(t)-\varepsilon\right),\\
\label{AftClm1Eq5b} N&=&-\frac{1}{2\phi(s)}(1,0,1)+\phi(s)\left(\gamma_0(t),\sqrt{-2\varepsilon \gamma_0(t)+1},\gamma_0(t)-\varepsilon\right)
\end{eqnarray}
and we have 
\begin{eqnarray}
\label{AftClm1Eq5c} e_2=\frac 1{a(t)s+b(t)} x_t&=&\left(\sqrt{-2\varepsilon \gamma_0(t)+1},-\varepsilon,\sqrt{-2\varepsilon \gamma_0(t)+1}\right).
\end{eqnarray}
\end{subequations}
By integrating \eqref{AftClm1Eq4}, we obtain
\begin{equation}\label{IntofAftClm1Eq4}
x(s,t)=\frac{1}{2\phi^2}(1,0,1)+s\left(\gamma_0(t),\sqrt{-2\varepsilon \gamma_0(t)+1},\gamma_0(t)-\varepsilon\right)+\Gamma(t)
\end{equation}
for a smooth $\mathbb E^3_1$-valued function $\Gamma$. By combining \eqref{IntofAftClm1Eq4} and \eqref{AftClm1Eq5c}, we get
$$\Gamma'(t)=\left(\left(a(t)-\frac{\gamma_0'(t)}{\sqrt{-2\varepsilon \gamma_0(t)+1}}\right)s+b(t)\right)\left(\sqrt{-2\varepsilon \gamma_0(t)+1},-\varepsilon,\sqrt{-2\varepsilon \gamma_0(t)+1}\right)$$
from which we conclude
$$a(t)=\frac{\gamma_0'(t)}{\sqrt{-2\varepsilon \gamma_0(t)+1}}$$
and
$$\Gamma'(t)=b(t)\left(\sqrt{-2\varepsilon \gamma_0(t)+1},-\varepsilon,\sqrt{-2\varepsilon \gamma_0(t)+1}\right).$$
By combining the last equation with \eqref{IntofAftClm1Eq4}, we obtain \eqref{ClassThmDiagonLightlikekSurf}. Hence, the proof of the necessary condition is completed.

Conversely, consider the surface $M$ given by \eqref{ClassThmDiagonLightlikekSurf} whose derivative does not vanish. A direct computation yields that unit normal of $M$ is
$$N=\left(\frac{\varepsilon   }{2 \phi (s)}-\varepsilon    \phi (s) \gamma _0(t),-\varepsilon    \phi (s) \sqrt{1-2 \varepsilon    \gamma _0(t)},\varepsilon    (-\phi (s)) \gamma _0(t)+\frac{\varepsilon   }{2 \phi (s)}+\phi (s)\right)$$
and the principle curvatures of $M$ are 
\begin{equation}\label{ClassThmDiagonLightlikekPrncplDrct}
e_1=\phi(s)\frac\partial{\partial s }\quad\mbox{and}\quad e_2=\frac {\sqrt{-2\varepsilon \gamma_0(t)+1}}{{s\gamma_0'(t)}+b(t)\sqrt{-2\varepsilon \gamma_0(t)+1}}\frac\partial{\partial t}.
\end{equation}
Moreover, we have $\langle e_2,(1,0,1)\rangle=0$ which yields that $(1,0,1)^T$ is a principle direction. Hence the proof of sufficient condition is completed.
\end{proof}
By considering the surface proof of Theorem \ref{ClassThmDiagonLightlikek}, we obtain the following proposition.
\begin{prop}\label{PropClassThmDiagonLightlikekShapeOp}
Let $M$ be the surface given by \eqref{ClassThmDiagonLightlikekSurf}. Then, the matrix representation of the shape operator $S$ of $M$ with respect to $\{e_1,e_2\}$ is
\begin{equation}\label{ClassThmDiagonLightlikekShapeOp}\displaystyle
S=\left(
\begin{array}{cc}
 \varepsilon    \phi '(s) & 0 \\
 0 & \frac{\varepsilon    \phi (s) \gamma _0'(t)}{\sqrt{1-2 \varepsilon    \gamma _0(t)} b(t)+s \gamma _0'(t)} \\
\end{array}
\right),
\end{equation}
where $e_1,\ e_2$ are vector fields given by \eqref{ClassThmDiagonLightlikekPrncplDrct}
\end{prop}
From Proposition \ref{PropClassThmDiagonLightlikekShapeOp} we conclude the following characterization results.
\begin{cor}
A flat surface  with diagonalizable shape operator in $\mathbb E^3_1$ endowed with CPD relative to a light-like direction is congruent to the surface given by
\begin{align}
\begin{split}
x(s,t)=&\left(c s+\sqrt{1-2 c \varepsilon   } t+\int_{ {s_0}}^s \frac{1}{2 \phi (\xi)^2}  d\xi,-\varepsilon   t  +s \sqrt{1-2 c \varepsilon   },\right.\\&\left.
s (c-\varepsilon  )+\sqrt{1-2 c \varepsilon   } t+\int_{ {s_0}}^s \frac{1}{2 \phi (\xi)^2}   d\xi\right)
\end{split}
\end{align}
for some constants $s_0, \varepsilon\in\{ -1,1\}$ and a non-vanishing function $\phi$ whose derivative does not vanish.
\end{cor}

\begin{cor}
A minimal (resp. maximal) surface with diagonalizable shape operator in $\mathbb E^3_1$ endowed with CPD relative to a light-like direction is congruent to the surface given by
\begin{align}
\begin{split}
x(s,t)=&\left(\frac{(c_1+s)^3}{c_2}+c_1 t,(s-c_1) \sqrt{1-2 t \varepsilon   },\frac{(c_1+s)^3}{c_2}+c_1 t+s (t-\varepsilon   )\right)
\end{split}
\end{align}
for some constants $c_2>0,c_1$ with $\varepsilon= -1$ (resp. $\varepsilon= 1$).
\end{cor}

\section*{Acknowledgments}
This paper is a part of PhD thesis of the first named author who is supported by The Scientific and Technological Research Council of Turkey (TUBITAK) as a PhD scholar.

\end{document}